\documentclass[12pt]{amsart}

\oddsidemargin=\dimexpr(\paperwidth-\textwidth)/2-1in\relax
\evensidemargin=\oddsidemargin

\RequirePackage[nonfrench]{kotex}
\usepackage[utf8]{inputenc}
\usepackage{kotex}

\usepackage{epsfig}

\usepackage{amsfonts}
\usepackage{amssymb}
\usepackage{amsmath}
\usepackage{amsthm}
\usepackage{amscd}
\usepackage{amstext}
\usepackage{latexsym}
\usepackage{amsmath,bm}

\usepackage{array}
\usepackage{graphicx, subcaption}
\usepackage{todonotes}

\usepackage{enumerate}

\usepackage{tikz-cd}

\usepackage[all]{xy}

 \numberwithin{equation}{section}

\usepackage[backref,bookmarks=false]{hyperref}

\usepackage{enumerate}

\newcommand{\CC}{\mathbb{C}}

\newcommand{\RR}{\mathbb{R}}

\newcommand{\JJ}{\mathcal{J}}

\newcommand{\iddb}{\sqrt{-1} \partial \overline{\partial}}

\newcommand{\yr}{ Y_{\reg}}

\providecommand{\abs}[1]{ |#1|}

\DeclareMathOperator{\reg}{reg}

\makeatletter
\newcommand{\sumprime}{\if@display\sideset{}{'}\sum%
	\else\sum'\fi}
\makeatother

\newcommand{\db}{\overline{\partial}}
\newcommand{\II}{\mathcal{I}}

\newcommand{\OO}{\mathcal{O}}
\newcommand{\al}{\alpha}

\newcommand{\mfa}{\mathfrak{a}}

\newcommand{\qa}{\quad}

\newcommand{\vp}{\varphi}

\newcommand{\lan}{\langle}
\newcommand{\ran}{\rangle}

\newcommand{\noi}{\noindent}

\providecommand{\abs}[1]{ |#1|}

\theoremstyle{plain}

\newtheorem{theorem}{Theorem}[section]

\newtheorem{lemma}[theorem]{Lemma}
\newtheorem{corollary}[theorem]{Corollary}
\newtheorem{proposition}[theorem]{Proposition}
\newtheorem{definition}[theorem]{Definition}

\newtheorem{question}[theorem]{Question}

\newtheorem{remark}[theorem]{Remark}

\theoremstyle{remark}

\newtheorem{remark1}[theorem]{Remark}

\DeclareMathOperator{\loc}{loc}

\def\loccoh#1.#2.#3.#4.{H^{#1}_{#2}(#3,#4)}

\DeclareMathAlphabet{\mathchanc}{OT1}{pzc}%
                                {m}{it}

\usepackage[all]{xy}\xyoption{dvips}

\providecommand{\norm}[1]{\lVert#1\rVert}

\begin{document}

	\title[$L^2$ extension in the general setting]{$L^2$ extension of holomorphic functions  \\ in the general non-analytic setting}

\author{Dano Kim and Xu Wang}

\date{\today}

\begin{abstract}
	In an influential $L^2$ extension theorem due to Demailly, the subvariety to extend from is defined by the multiplier ideal of a quasi-plurisubharmonic defining function with analytic singularities. In his theorem, the finiteness of an $L^2$ norm called the Ohsawa norm plays the crucial role of  determining whether a given holomorphic function or section can be extended.   
In  2016, Demailly asked if the Ohsawa norm  behaves equally well when the defining function is further generalized to admit non-analytic singularities. 

We answer this question negatively by showing that for a wide class of log canonical defining functions with a unique log canonical place, only the zero holomorphic function on the subvariety has finite Ohsawa norm. It follows that for such a defining function, the $L^2$ extension theorem is vacuous. 
This reveals a fundamental obstruction in the  theory of $L^2$ extension in the general non-analytic setting. 

  \end{abstract}
\maketitle


	\section{Introduction}

  Let $Y \subset X$ be a submanifold of a complex manifold. Let $L$ be a holomorphic line bundle on $X$ and $K_X$  the canonical line bundle of $X$.  An \emph{$L^2$ extension theorem}
  asserts that, under suitable hypotheses
   on $X, Y, L, \ldots$,  if a certain $L^2$ norm $ \norm{s}_Y$   is finite for a holomorphic section $s$ on $Y$ of  $(K_X + L)|_Y$, then there exists  $\tilde{s} \in H^0(X, K_X + L)$   such that  $\tilde{s}|_Y = s \; \text{ and } \; \norm{\tilde{s}}_X \le c \norm{s}_Y $ for some constant $c>0$.  
 Since the groundbreaking paper  \cite{OT87}, there have been extensive developments on $L^2$ extension theorems: we refer to \cite{O94, M93, O01, D00, S02, K10, B13, GZZ12, D15, CWW15, O15, NW22} for some of them. 
 
 	\subsection{$L^2$ extension theorems formulated in terms of the Ohsawa norm}{\ }

  In recent developments, some general versions of $L^2$ extension theorems have been formulated in terms of a quasi-plurisubharmonic function on $X$, say $\Psi$, as in the
	following result originally due to Demailly~\cite[Theorem 2.8]{D15} when $\Psi$ has neat analytic singularities.

	\begin{theorem}\label{th:ZZ0} \cite{D15}, \cite{ZZ22}
	Let $X$ be a weakly
		pseudoconvex K\"ahler manifold of dimension $n \ge 1$.  Let $(L, h)$ be a smooth hermitian line bundle on $X$. 
			Let $\Psi \le 0$ be a quasi-plurisubharmonic function on $X$ such that it is log canonical at every point of $X$ (i.e. the log canonical threshold $c_x (\Psi) \ge 1$ for every $x \in X$). Assume the following curvature conditions for some $\al > 0$: 
			
\begin{align}\label{curva}
   \sqrt{-1} \Theta(L, h) +  \iddb \Psi  \ge 0 \text{\qa and \qa} 
      \sqrt{-1} \Theta(L, h) +  (1+\alpha) \iddb \Psi  \ge 0.
\end{align}
	Let $Y \subset X$ be the reduced (possibly reducible) subvariety defined by the multiplier ideal sheaf of $\Psi$, assuming that $Y$ is nonempty. If a holomorphic section $f \in H^0 ( Y, (K_X + L)|_Y)$ has finite Ohsawa norm i.e.
	\begin{equation}\label{Onorm}
	||f||_\Psi < \infty,
	\end{equation}
	 then there exists a holomorphic section $F \in H^0 (X, K_X+ L)$ satisfying $F|_Y = f$ and 

\begin{equation}\label{eq:ZZ1.2}
\int_X \abs{F}^2_{\omega, h} dV_{X, \omega} \le  \frac{1 + \al}{\al} ||f||^2_\Psi. 
\end{equation}
	
	\end{theorem}
	
	We  call $\Psi$ a \textit{defining function} for the subvariety $Y$ in the sense that $Y$ coincides with the non-integrable locus of $\Psi$.  The defining function replaces more usual  holomorphic defining function for $Y$ as in \cite{OT87}, \cite{M93}. Note that $Y$ is contained in, but not necessarily equal to the pole set $\Psi^{-1} (-\infty)$ of $\Psi$. Consequently,  this affords us substantial flexibility in the choice of a defining function for a given $Y$.
		
The Ohsawa norm $||f||_\Psi$ in \eqref{Onorm} is defined as follows. Take $F$ to be a $C^{\infty}$ extension of $f$ (as a $C^{\infty}$ section of $K_X+L$) using local extensions and a partition of unity. We define (where $h = e^{-\phi}$) the squared Ohsawa norm in Theorem~\ref{th:ZZ0} to be 
	
		\begin{equation}\label{eq:ZZ11}
		||f||^2_\Psi :=	||F||^2_\Psi :=\sup_{K \ \text{compact in} \ X}\limsup_{t\to -\infty} \int_{\{t<\Psi<t+1\} \cap K} i^{n^2} F\wedge \overline{F} \,e^{-(\phi+\Psi)}
		\end{equation}
which is indeed well defined independent of the choice of $F$.  This definition gives the connecting link from the condition \eqref{Onorm} to the solvability of the $\db$-equation for the extension.

At this level of generality, establishing the finiteness $\|f\|_\Psi < \infty$ is far more non-trivial than in the original setting of \cite{OT87} owing to the \textit{singularities} of the Ohsawa norm, which we will define shortly. A thorough understanding of these singularities is indispensable, given that the $L^2$ extension  requires the Ohsawa norm to be finite.  We emphasize that when $\Psi$ has non-analytic singularities, log resolutions are not available in general. Therefore one cannot use the method due to \cite{D15} of understanding the Ohsawa measure as a pushforward measure from the resolution.  Instead, one only knows the mere analytic existence of the Ohsawa norm \eqref{eq:ZZ11}. 

 In fact, to our knowledge, even the following basic question seems to have been unanswered in the literature: does there exist a single example of $X, \Psi$ and $L$ as in Theorem~\ref{th:ZZ0} with $\Psi$ having non-analytic singularities such that the condition $||f||_\Psi < \infty$ is satisfied by at least one holomorphic function $f$ not identically zero? 

This is related to the following question of Demailly from his lecture at Collège de France in 2016~\cite{D16}: 

\begin{question}\label{DemQ}

Does the Ohsawa norm  `behave equally well' when the defining function is generalized to admit non-analytic singularities—that is, as well as it does in the analytic case? 

\end{question}

Here, `good behavior' means that the singularities of the Ohsawa norm are confined to a proper subset of $Y$, ensuring that at least some non-zero holomorphic functions possess a finite Ohsawa norm. This condition is essential, as it is the only scenario in which the $L^2$ extension theorem carries meaningful content. 

 The other case of `bad behavior' means that the singularities of the Ohsawa norm occur along the entire $Y$, so that no holomorphic function, except constant zero, has finite Ohsawa norm.   In such a case, the $L^2$ extension theorem is  entirely vacuous, as there are no applicable functions.

In the following section, we will make the question of Demailly more precise by introducing necessary notions, see Question~\ref{Ques}. 

\subsection{Singularities of the Ohsawa norm}		{\ }

Let $X$ be a complex manifold and $\Psi$ a quasi-psh function on $X$. We will say that the pair $(X, \Psi)$ is \textbf{log canonical} if the log canonical threshold of $\Psi$ is  $\ge 1$ at every point of $X$.

From now on, we will denote  the non-integrable locus of $e^{-\Psi}$ by $Y_0$ instead of $Y$. 
In the  setting of Theorem~\ref{th:ZZ0}, 
 we  say that the Ohsawa norm of $\Psi$ is \textbf{singular} at a point $x \in Y_0$ if $||f||_\Psi=\infty$ for every   $f$ defined near $x$  with $f(x)\neq0$.  Note that the term `singular' here refers to \emph{non-integrable singularity of the density function} of the corresponding Ohsawa measure.
 It can also be viewed as a \emph{residual singularity} left by $\Psi$ on $Y_0$.

For $f$ to be extended by the $L^2$ extension theorem,  $f$ has to vanish at all such singular points. This  imposes a crucial necessary condition for extension.

 Therefore, given $\Psi$, one needs  to understand the \textbf{singular locus} of the Ohsawa norm of $\Psi$, i.e. the subset of all the singular points.  Let us denote it by $Z := Z(\Psi) \subset Y_0 = \{x\in X: c_x(\Psi) = 1 \}$
	where we now denote the non-integrable locus of $e^{-\Psi}$ by $Y_0$.     Recall that   the log canonical threshold of $\Psi$ at a point $x \in X$ is defined  as 
	\begin{equation}
		c_x(\Psi):=\sup\{c\geq 0: e^{-c\Psi} \ \text{is locally integrable at $x$}\}.
	\end{equation}

	One extreme possibility is when the singular locus $Z$ fills up the entire $Y_0$ : in such a case, the corresponding case of Theorem~\ref{th:ZZ0} is vacuous since only the zero function can satisfy the finiteness of the Ohsawa norm.  \footnote{Possibility of such phenomenon was first discovered and pointed out by Chen-Yu Chi, cf. \cite{K21}.}
	We will say that the Ohsawa norm is \textbf{entirely singular} if $Z = Y_0$. Clearly this is the worst case of $\Psi$ which should be excluded if one wants to have any application of Theorem~\ref{th:ZZ0}. 
	
More generally, it is natural to study the singularity of the Ohsawa norm on each irreducible component, say $Y$, of $Y_0$. We will say that the Ohsawa norm is \textbf{entirely singular on $Y$} if $Z \cap Y =  Y$.

 When $\Psi$ has analytic singularities, the work of \cite{K21} identified when `entirely singular' happens in a precise manner:   in that case, the singularities of $\Psi$ can be resolved due to Hironaka's theorem and the non-integrable singularities of $\Psi$ are contributed precisely by divisorial valuations with discrepancy $-1$. We call such a valuation  a \textbf{log canonical place} of the pair $(X, \Psi)$, following the usual terminology in algebraic geometry, cf. \cite{Ko13}.  In terms of this notion, \cite{K21} showed that the Ohsawa norm  is not entirely singular on $Y$ precisely when $Y$ has only one log canonical place  (cf. Theorem~\ref{K21}).
 
 In recent developments,    the notion of a log canonical place now makes sense even for non-analytic singularities within a broader class than divisorial valuations, namely in the class of quasimonomial valuations (cf.  \cite{KK23} based on the works of \cite{JM12}, \cite{JM14}, \cite{X20}).  For $(X, \Psi)$ a log canonical pair, a quasimonomial valuation $v$ on $X$ (up to positive scaling)  is called a \textbf{{log canonical place}}  for $(X, \Psi)$ if $A(v) - v(\Psi) = 0$ where $A(v)$ is the log discrepancy of $v$ (cf. \cite{BFJ08}), see Definition~\ref{lcplace}.   Throughout this paper, a unique log canonical place means unique up to positive scaling.

  Based on these preparations, 
 Demailly's Question~\ref{DemQ} is made precise as follows:

  \begin{question}\label{Ques}
  
 Let $X$ be a complex manifold. 
 \begin{itemize}
 \item
 
 Let $(X, \Psi)$ be a log canonical pair with $\Psi$ a quasi-plurisubharmonic function with non-analytic singularities. 
 
 \item Let $Y$ be an irreducible component of the non-integrable locus of $\Psi$ such that it has a unique log canonical place.
 
 \item Let $Z(\Psi)$ be the singular locus of the Ohsawa norm of $\Psi$. 
 
 \end{itemize}
 
  Is $Z(\Psi)\cap Y$ always a proper subset of $Y$? 
  
 \end{question} 
 
\noi Recall that when $\Psi$ has analytic singularities, this is indeed the case, cf.\cite{D15, K21}.

  \subsection{Main results} {\ }

 We answer Question~\ref{Ques} negatively by  the following theorem.

	\begin{theorem}\label{maincoro}
	There exists $(X, \Psi)$ and $Y$ as in Question~\ref{Ques} such that
	the Ohsawa norm of $\Psi$ is entirely singular on $Y$. 	
	\end{theorem}

 It follows that Theorem~\ref{th:ZZ0} is vacuous in the corresponding case due to this bad behavior.  There does not seem to be a clear-cut standard to distinguish between the good and the bad behaviors, cf. Corollary~\ref{notvoid}, Remark~\ref{almost}.

Theorem~\ref{maincoro} is a consequence of the following main result which characterizes the singularity of the Ohsawa norm in terms of the Newton convex body when  $\Psi$ is toric plurisubharmonic with isolated singularities. In fact, we take $\Psi$ to be $P$-toric for a Newton convex body $P \subset \RR^n_{\ge 0}$, i.e. $\Psi$ is a  toric psh function whose singularity is determined, up to a bounded term, by its Newton convex body : see Definition~\ref{toric}.

	\begin{theorem}\label{th:main} 
	Let $P \subset \RR^n_{\ge 0}$ be a Newton convex body. 
	Let $\Psi$ be a $P$-toric (see Definition~\ref{toric}) psh function on the unit polydisc $\mathbb D^n \subset \mathbb C^n$ with isolated singularities at $0 \in \mathbb D^n$. Assume that $c_0(\Psi)=1$. Then the following are equivalent:
		\begin{enumerate}
			\item[(1)] The Ohsawa norm of $\Psi$ is not singular at the origin $0 \in \mathbb D^n$. 
			\item[(2)]  The boundary $\partial P$ of $P$ is locally a hyperplane near the point  $(1,\cdots, 1) \in \partial P$. \footnote{Note that (cf. \cite{R13}, \cite{G12}) $(1, \cdots, 1) \in \partial P$ if and only if  
$c_0(\Psi)=1$.} In other words,  there exists an open ball $B $ centered at $(1,\cdots, 1)$ such that 
			$$
			\{x\in B: \alpha_1 (x_1-1)+\ldots + \alpha_n (x_n-1) \geq 0\} = P \cap B
			$$
			for some $\alpha_j \geq 0$, $1\leq j\leq n$, with $\alpha_1+\cdots+\alpha_n=1$.

	\end{enumerate} 
	
	\end{theorem}

It is remarkable that  the condition (2) corresponds to a very special case  among all the possible shapes of   $\partial P$, the boundary of a Newton convex body. In this sense,  one observes that `almost all' the cases of $\Psi$ are entirely singular in this setting. 

We deduce Theorem~\ref{maincoro} from Theorem~\ref{th:main} as follows.   As is well known (cf. \cite{R13}, \cite{KS20}), given an arbitrary Newton convex body $P \subset \RR^n_{\ge 0}$, there exists a toric psh function (non-uniquely) having $P$ as its Newton convex body. Therefore Theorem~\ref{th:main} immediately produces a toric psh function $\Psi$ whose Ohsawa norm is singular at the origin by taking $P$ a Newton convex body not satisfying (2). We also require  $\partial P$ to be a smooth real hypersurface with the tangent hyperplane $H$ at $(1,\cdots, 1)$.  Then the Ohsawa norm of $\Psi$ is  entirely singular having $Y_0 = Z = \{ 0 \}$, with a unique log canonical place given by the unique supporting hyperplane for $P$, namely $H$.

	Another consequence of Theorem~\ref{th:main} establishes the following, which answers the question raised before Question~\ref{DemQ}. 
	
	\begin{corollary} \label{notvoid}
	
	There exists $(X, \Psi)$ and $Y$ as in Question~\ref{Ques} such that
	the Ohsawa norm of $\Psi$ is not entirely singular on $Y$, so that 	Theorem~\ref{th:ZZ0} is not vacuous in the corresponding case. 
	
	\end{corollary} 
	
	 We emphasize again that this is in a much more general setting than \cite{D15} which assumed analytic singularities, due to the absence of a log resolution for $\Psi$ with non-analytic singularities in general.

	The method of proof for Theorem~\ref{th:main} is in the spirit of computing directly the integral appearing in the definition of the Ohsawa norm, which can be indeed computed in many concrete cases.  For a toric psh function, we transition from the finiteness of the integral to another kind of integral on the associated  Newton convex body, see Proposition~\ref{pr:toric}. Then we analyze the finiteness of the transitioned integral using     convex analysis.

\begin{remark1}\label{almost}

It is natural to expect that the above phenomenon of \textit{`almost all' the cases of $\Psi$ are entirely singular} will be also observed for more general non-toric $\Psi$'s. On the other hand, the condition (2) of Theorem~\ref{th:main} is in a form  specific to the toric psh case.
\end{remark1}

The structure of this paper is as follows. In Section 2, we define monomial valuations in our analytic setting and discuss monomial valuations computing the  log canonical thresholds of  toric plurisubharmonic functions. In Section 3, we define the Ohsawa measure and the Ohsawa norm and establish their basic properties. In Section 4, we prove the main results Theorem~\ref{maincoro} and Theorem~\ref{th:main}. 
\\

	 \noi \textbf{Acknowledgment.}
	 DK  was supported by the National Research Foundation of Korea (NRF) grant funded by the Korean government  (No. 2019R1A6A1A10073437).

	 	\section{Monomial valuations and toric plurisubharmonic functions} 	
	
For introduction to plurisubharmonic (or psh for short)  functions, we refer to  \cite{Ks93, D11, DK01}.

A psh (or quasi-psh) function $\vp$ on a complex manifold $X$ has \textbf{analytic singularities} if there exist a coherent ideal sheaf $\mfa \subset \OO_X$ and a real constant $c \ge 0$ such that  $\vp$ is locally of the form $\vp = c \log \sum^{m}_{i=1} \abs{g_i}^2 + u$ where  $g_1, \ldots, g_m$ are local holomorphic functions generating $\mfa$ and $u$ bounded. (We will say that $\vp$ has neat analytic singularities if such $u$ can be chosen as $C^{\infty}$.) 
In this case, one also says $\vp$ has analytic singularities of type  $\mfa^c$. 

 A \emph{log resolution} of a quasi-psh function $\vp$ with analytic singularities of type $\mfa^c$  is a proper modification   $f: X' \to X$ such that the pullback $f^* \vp$ has divisorial poles along a divisor $G$ on $X'$ and $G+F$ is a simple normal crossing divisor where $F$ is the exceptional divisor of $f$. Such a log resolution exists since a principalization of the ideal sheaf $\mfa$ exists by Hironaka~\cite{H64}, cf. \cite{BM97}.

On the other hand, if a quasi-psh $\vp$ does not have analytic singularities, we will say that $\vp$ has non-analytic singularities, or  simply that $\vp$ is \textbf{non-analytic}.

 Next we recall  the following generalization of the Lelong number, cf.\ \cite{Ks94}.

\subsection{Kiselman number} {\ }

Now let $\vp$ be a quasi-psh function (not necessarily with analytic singularities) defined near the origin $0 \in \CC^n$ with  coordinates $(z_1, \ldots, z_n)$. 

Let $\al =(\al_1, \ldots, \al_n) \in \RR^n_{\ge 0} \setminus\{0\}$. 
The \textbf{Kiselman number} of $\vp$ at $0 \in \CC^n$ with the weight $\al$  is given by 

\begin{equation}\label{Kisel}
 v_\al (\vp) := \sup \Bigl\{ t \ge 0 : \vp \le t \log \Bigl( \max_{1 \le j \le n, \al_j > 0} \abs{z_j}^{\frac{2}{\al_j}} \Bigr) + O(1) \Bigr\}. 
\end{equation}

\noi 
When $\alpha=(1,\ldots,1)$, we have
$v_\alpha(\varphi)=\frac12\nu(\varphi,0)$,
where $\nu(\varphi,0)$ denotes the usual Lelong number.  See \cite{Ks93, Ks94} for basic properties and other characterizations of the Kiselman number.

When we take $\vp = \log \abs{f}^2$ for a holomorphic function germ $f \in \OO_{\CC^n, 0}$, we see that $$v_\al (f) := v_\al (\log \abs{f}^2)$$ defines a valuation on the local ring $\OO_{\CC^n, 0}$  in that $v_\al (fg) = v_\al (f) + v_\al (g)$ and $v_\al (f+ g) \ge \min (v_\al (f), v_\al (g))$.

\subsection{Monomial valuations} {\ }

Now let $X$ be a complex manifold and $D = \sum_{i \in I} D_i$ a reduced simple normal crossing divisor on $X$. Let $Z$ be a stratum of $D$, i.e. an irreducible component of $\bigcap_{j \in J} D_j$ for some $J \subset I$. For simplicity of notation, let $J = \{ 1, \ldots, m \}$ where $m$ is the codimension of $Z$. 

Let $p \in Z$ be a point and let $z=(z_1, \ldots, z_n)$ be a local coordinate  near $p$ such that the  divisor $D_j$ is defined by $z_j=0$ for $j=1, \cdots, m$. Let  $\al \in \RR^n_{\ge 0}$ be such that $\al_j >0$ for $j \le m$ and $\al_j = 0$ for $m+1 \le j \le n$.

Let $\vp$ be a quasi-psh function on $X$. We take the Kiselman number $v_\al (\vp, p)$ as in \eqref{Kisel} taking $p$ in the place of $0$.  

The important properties of Kiselman numbers given by \cite[Lem. 3.1]{JM14} implies that the Kiselman number $v_\al (\vp, p)$ is constant for $p \in Z$ (since $Z$ is connected) and that $v_\al (\vp, p)$ depends only on the divisor $D_1 +\cdots+ D_m$, not on the particular choice of analytic coordinates adapted to it. 
Due to these properties, we can denote this Kiselman number as $\kappa_{Z,D, \al} (\vp)$.

\begin{definition} 

Let $(X,D)$ and $Z$ be as above. A \emph{\textbf{monomial valuation}} $v$ on $X$ is a function from the set ${\rm QPSH}(X)$ of all quasi-psh functions on $X$

\begin{equation}\label{monoval}
 v: {\rm QPSH}(X)\to \RR_{\ge 0}
 \end{equation}

\noi  defined by $v(\vp) := \kappa_{Z, D, \al} (\vp) $.  We  call $Z$ the center of $v$ on $X$. 

\end{definition}

More generally one can define a \emph{quasimonomial valuation} on $X$ to be a similar function $v$ as \eqref{monoval} when $D$ is a simple normal crossing divisor on $Y$ with a proper modification $\mu: Y \to X$ and  $v(\psi)$ is taken to be $\kappa_{Z, D, \al} (\mu^* \psi)$. For the purpose of this paper, we primarily utilize monomial valuations. 

\begin{remark1}
These are geometric (and analytic) versions of the more usual notions of monomial and quasimonomial valuations for the function field of a complex algebraic variety, cf. \cite{FJ05}, \cite{JM14}, \cite[Cor. 5.4]{KK23}. In particular, we can use these geometric versions when $X$ is a compact complex manifold whose function field of global meromorphic functions might be too small. 
\end{remark1}

\subsection{Toric plurisubharmonic functions} {\ }

	 Let $\vp$ be a plurisubharmonic function on the unit polydisc $\mathbb D^n \subset \mathbb C^n$ centered at $0$. We say that $\vp$ is \textbf{toric} if
		$$
		\vp(z) := \vp(z_1, \cdots, z_n) =\vp(|z_1|,\cdots, |z_n|), \ \ \forall \ z\in \mathbb D^n.
		$$

	It is known that (cf. \cite{Ks94}), to such $\vp$ is associated a convex function $g= g(x_1, \cdots, x_n)$ on $(-\infty, 0)^n$ increasing in each variable such that $$\vp(z)= g(\log|z_1|^2, \cdots, \log|z_n|^2)$$ holds (see also \cite{R13}, \cite{G12}, \cite[Sec.2]{AS23} for more exposition on toric psh functions).
		Let us call a closed convex subset $P$ of $[0,\infty)^n = \RR^n_{\ge 0}$ a \textbf{Newton convex body} if $P+[0,\infty)^n=P$. 
		
		\begin{definition}\label{Newton-convex-body} Let $\varphi$ be toric plurisubharmonic on $\mathbb D^n$ with associated convex increasing function $g$ on $(-\infty, 0)^n$. The Legendre transform of $g$ is defined by
			$$
			g^*(\alpha):=\sup_{x\in (-\infty, 0)^n} \, \{x\cdot \alpha-g(x)\},
			$$
			We call the closure of the finite locus of $g^*$
			$$
		P(\varphi):=P(g):=\overline{\{g^*<\infty\}}	
			$$
			the Newton convex body of $\varphi$.
\end{definition}

It is clear that all $P(g)$ above are Newton convex bodies. On the other hand, if $P$ is a Newton convex body then we have $P(\varphi_P)=P$, where 
\begin{equation}\label{eq:phiP}
\varphi_P(z):=h_P(\log|z_1|^2, \cdots, \log|z_n|^2), \ \ \ h_P(x):=\sup_{\alpha\in P} \alpha\cdot x.
\end{equation}
For a general toric plurisubharmonic $\varphi$, $\varphi-\varphi_{P(\varphi)}$ can be unbounded on $\mathbb D^n$ (one may simply look at $$\varphi(z):=\max\{0, \log|z|\} -\log(-\log|z|)$$ on the unit disk $|z|<1$ then $$\varphi_{P(\varphi)}=\max\{0, \log|z|\}$$ and $\varphi-\varphi_{P(\varphi)}=-\log(-\log|z|)$ is unbounded on $|z|<1$). This phenomenon suggests the following definition (cf. \cite{LR99}). 
		
		\begin{definition}\label{toric}
		Let $P$ be a Newton convex body and $\varphi$ be  toric plurisubharmonic on $\mathbb D^n$. We say that $\varphi$ is \textbf{\emph{$\bm P$-toric}}  if
		$P(\varphi)=P$ and $\varphi-\varphi_{P(\varphi)}$ is bounded on $\mathbb D^n$.
	\end{definition}

	It is known that many singularity invariants of a toric plurisubharmonic $\vp$ such as multiplier ideals of all $m\vp$ ($m > 0$) are described in terms of its Newton convex body $P(\vp)$, cf. \cite{R13, G12, JM12, AS23}.

	\begin{remark1}
	The function $\varphi_{P(\vp)}$ is the local indicator of $\vp$
at the origin in the sense of Lelong--Rashkovskii, see [LR99].
Thus the $P$-toric condition requires that $\vp$ differ from
its local indicator by a bounded function. 
	\end{remark1}

 \subsection{Monomial valuations computing the log canonical threshold} {\ }
 
 When a plurisubharmonic function is toric, we have the following valuative characterization of the log canonical threshold in terms of monomial valuations due to \cite{Ks94} (also cf. \cite{G12}). 
  
 \begin{theorem}\label{valun}
 
 Let $\vp$ be a toric psh function (with respect to a choice of analytic coordinates $(z_1, \ldots, z_n)$) on the unit  polydisc centered at $0 \in \CC^n$.  We have
 
\begin{equation}\label{lct_inf} 
  c_0 (\vp) =   \inf_{v} \frac{A(v)}{v(\vp)} 
 \end{equation} 
 where $v=v_\alpha$ ranges over the monomial valuations
with respect to $(z_1,\ldots,z_n)$, with
$\alpha\in\mathbb R_{\ge0}^n\setminus\{0\}$.
Their centers contain $0$, but need not equal $\{0\}$.
Here $A(v_\alpha)=\sum_{j=1}^n\alpha_j$ is the log
discrepancy, and the quotient is understood to be
$+\infty$ when $v_\alpha(\varphi)=0$.
 \end{theorem} 

 If a monomial valuation $v_0$ achieves the infimum in \eqref{lct_inf}, we will say that $v_0$ computes the log canonical threshold. By \cite{JM12}, such a monomial valuation always exists (not necessarily uniquely) for a toric psh function. Moreover, if a quasimonomial valuation computes the log canonical threshold, then it should be a monomial valuation:

\begin{theorem}\label{theonly}

Let $\vp$ be a toric plurisubharmonic function on the unit polydisc $\mathbb D^n \subset \CC^n$ centered at $0$. Assume that $0<c_0(\varphi)<\infty$.

\begin{enumerate}
\item
 There exists a monomial valuation computing the log canonical threshold of $\vp$ at $0$.

\item

If a quasimonomial valuation computes the log canonical threshold, then it is equivalent to a monomial valuation. 

\end{enumerate}

\end{theorem} 

 This is an analytic version of \cite[Prop. 8.1]{JM12}  for a graded sequence of monomial ideals which can be easily converted to this analytic version.  The equivalence in (2) refers to the equality of functions as in \eqref{monoval} defining the valuations.

 Such a monomial valuation computing the log canonical threshold is characterized in terms of supporting hyperplanes on the associated Newton convex body as follows.

  In $\RR^n$, we say that a hyperplane $H$ supports a subset $A \subset \RR^n$ at a point $x \in A \cap H$ if $A$ is contained in one of the two closed half-spaces determined by $H$. We will say that $H$ is a \textit{supporting hyperplane}  of $A$ if $H$ supports $A$ at some point of $A \cap H$. When $A$ is a nonempty closed convex subset of $\RR^n$, it is a basic fact (cf. \cite[(1.3.2)]{Sc13}) that for each boundary point $x$ of $A$, there exists a supporting hyperplane of $A$ at $x$. We will take $A$ to be a Newton convex body (which is closed convex but not bounded). 

\begin{proposition}\label{suphyp}
Let $\vp$ be a toric plurisubharmonic  function on the unit polydisc $\mathbb D^n$ centered at $0 \in \CC^n$.  Let $c := c_0 (\vp) \in (0, \infty)$ be the log canonical threshold at $0$. 
 There is a one-to-one correspondence between
 
\begin{itemize}
\item the set of  monomial valuations computing $c$, up to positive scaling, and 

\item the set of supporting hyperplanes at $\kappa (1, \ldots, 1)$ of the Newton convex body $P = P(\vp)$ where $\kappa = \frac{1}{c}$.

\end{itemize}

\end{proposition}

\begin{proof} 

Let $v_\al$ be a monomial valuation with $\al = (\al_1, \ldots, \al_n)$.  If $v_\al$ computes the log canonical threshold, we have 

$$ c = \frac{\al_1+ \ldots + \al_n}{v_\al (\vp)} $$
From Lemma~\ref{le:Kisel}, for every $x \in P$, we have $\lan x, \al \ran \ge v_\al (\vp) = \kappa (\al_1 + \ldots + \al_n)$. Hence $P$ lies on one side of the hyperplane $\lan x- \kappa (1, \ldots, 1), \al \ran =0$.

Conversely, let $\lan x- \kappa (1, \ldots, 1), \al \ran =0$ be a supporting hyperplane so that $\lan x, \al \ran  \ge \kappa (\al_1 + \ldots + \al_n)$ holds for every $x \in P$. By the same Lemma~\ref{le:Kisel}, it is easy to see that $v_\al$ computes the log canonical threshold. 
\end{proof} 

\begin{lemma}\label{le:Kisel} 
Let $\vp$ be a toric plurisubharmonic  function in the unit polydisc $\mathbb D^n$  centered at $0 \in \CC^n$. Let $v_\al (\vp)$ be the monomial valuation of $\vp$ with weight $\al = (\al_1, \ldots, \al_n)$. Then we have 
\begin{equation} \label{wehave}
 v_\al (\vp) = \inf \{ \lan x, \al \ran : x \in P(\vp) \}. 
 \end{equation}
\end{lemma}

\begin{proof} 

Let $\vp_m$ be the Demailly approximation~\cite{D92} of $\vp$. Each $\vp_m$ is given by a monomial ideal for which the conclusion holds. We have the convergence $v_\al (\vp_m) \to v_\al (\vp)$ as $m \to \infty$, cf. \cite{BFJ08}.    On the other hand,  the Newton convex body of $\vp_m$ converges to that of $\vp$: hence the right hand side of \eqref{wehave} also converges.  
\end{proof} 

\subsection{Log canonical places}  {\ }

Let $X$ be a complex manifold and $\vp$ a quasi-psh function on $X$. We will say that the pair $(X, \vp)$ is \textbf{log canonical} if the log canonical threshold of $\vp$ is  $\ge 1$ at every point of $X$. 

\begin{definition}\label{lcplace}
For $(X, \vp)$ a log canonical pair, a quasimonomial valuation $v$ on $X$ (up to positive scaling)  is called a \textbf{\emph{log canonical place}}  for $(X, \vp)$ if $A(v) - v(\vp) = 0$ where $A(v)$ is the log discrepancy of $v$ (cf. \cite{BFJ08}). 
\end{definition}

Proposition~\ref{suphyp} produces many examples of toric psh functions with non-unique monomial valuations computing the log canonical threshold. 

Especially in the log canonical case, it gives numerous examples of toric psh functions $\vp$ with analytic singularities with non-unique log canonical places when we take  a Newton convex body being polyhedral with a `vertex' at $(1, \ldots, 1)$ which has infinitely many supporting hyperplanes through the vertex. We note that this adds many more concrete examples of $\vp$ with analytic singularities with non-unique log canonical places   to the one due to O. Fujino provided in \cite[Example 2.10]{K21}.

\begin{remark1}\label{noteven}

In Theorem~\ref{th:main}, note that it is possible that the boundary of $P$ is a hyperplane in a small open neighborhood of the point $q := (1, \cdots, 1)$, however not even polyhedral somewhere away from $q$. In such a case, the corresponding toric psh function $\Psi$ has non-analytic singularities. Indeed, if it has analytic singularities, it should be of the form $\Psi = c \log \abs{\mfa} + O(1)$ for a monomial ideal $\mfa$ and $c > 0$ due to \cite[Prop. 3.1]{AS23}. Then $P$ should be a positive multiple of the Newton polyhedron associated to $\mfa$, which contradicts the fact that $P$ is not polyhedral somewhere away from $q$. 
\end{remark1}

	\section{$L^2$ extension and the Ohsawa norm}

	We recall the notions of the Ohsawa measure and the Ohsawa norm which play a key role in formulating and understanding general $L^2$ extension theorems. Then we define and study the singularity of the Ohsawa norm. 
	
	\subsection{Ohsawa measure and Ohsawa norm} {\ }
	
	Demailly's $L^2$ extension theorem for $\Psi$ a quasi-psh function with log canonical analytic singularities was formulated in terms of the Ohsawa measure associated to $\Psi$. We recall its definition here.

 Let $X$ be a complex manifold and $\Psi$ a quasi-psh function with log canonical neat analytic singularities. Let $Y$ be a nonempty irreducible component of the non-integrable locus of $\Psi$, i.e. the closed analytic subset defined by the multiplier ideal sheaf $\JJ(\Psi)$ of $\Psi$. 

\begin{definition}\cite{O01}, \cite{D15} \label{Omeasure}
Let $dV_X$ be a smooth volume form on $X$. Let $\yr$ be the regular locus of $Y$, i.e. the set of nonsingular points. A positive measure $d\mu$ on $\yr$ is called  the  \textbf{\emph{Ohsawa measure}} $dV[\Psi]$ of $\Psi$ on $Y$ (with respect to $dV_X$) if it satisfies the following condition:  for every $u$, a real-valued compactly supported continuous function on $\yr$ and for every $\tilde{u}$, a compactly supported extension of $u$ to $X$, we have the relation 
 
 \begin{equation}\label{dvp}
 \int_{\yr} u \; d\mu = \lim_{t \to -\infty} \int_{\{x \in X, t < \Psi(x) < t+1 \}} \tilde{u} e^{-\Psi} dV_{X}. 
\end{equation}

\end{definition}

	This notion first appeared in \cite{O01} (cf. \cite{O94})  in the  case when $\Psi$ is of the form 
	$$
	\Psi= k \log (|z_1|^2+\cdots+|z_k|^2) + \text{continuous term},
	$$
	where $z_1, \cdots, z_k$ are the local defining functions of a closed $(n-k)$ dimensional complex submanifold $Y$ of $X$.  For more general $\Psi$ with log canonical analytic singularities, the existence  and basic properties of the Ohsawa measure is due to \cite{D15} (cf. the exposition in \cite{K21}, \cite{KS21}).

In the generality where $\Psi$ is allowed to have non-analytic singularities (also not necessarily log canonical), we bypass showing the existence of a measure and work with the Ohsawa norm defined as follows.
	
	\begin{definition}\label{de:Ohsawa-norm}

	Let $X$ be a complex manifold of dimension $n$.  Let $\Psi$ be a quasi-plurisubharmonic function on $X$. Let $g$ be a positive $(n,n)$-form with compact support in $X$ (positive in the natural orientation of $X$).  We call
		\begin{equation}\label{eq:Ohsawa-norm}
			||g||_\Psi:= \limsup_{t\to -\infty} \int_{\{t<\Psi<t+1\}} g \,e^{-\Psi}
		\end{equation}
		the \textbf{\emph{Ohsawa norm}} of  $g$.  More generally, when $g$ is not with compact support, we define 
	
	$$
	||g||_\Psi=\sup_{K \ \text{compact in} \ X} ||1_Kg||_\Psi,
	$$
	where $1_K$ is defined to be $1$ on $K$ and $0$ outside $K$.
	
		\end{definition}
	
	When the Ohsawa measure $dV[\Psi]$ in Definition~\ref{Omeasure} exists, taking $g := \tilde{u} dV_X$ in Definition~\ref{de:Ohsawa-norm} recovers the same norm $\int_{\yr} u \;  dV[\Psi]$ in \eqref{dvp}. 
	
	Using Definition~\ref{de:Ohsawa-norm}, we define the Ohsawa norm employed in formulating Theorem~\ref{th:ZZ0} following the treatment of  \cite{NW24}. In the setting of Theorem~\ref{th:ZZ0}, let $$f \in H^0 (Y, (K_X+ L)|_Y)$$ be a holomorphic section to be extended in Theorem~\ref{th:ZZ0}.   Let $(\chi_j)$ be a partition of unity subordinate to an open Stein covering $(U_j)$ of $X$.		Let $F_j$ be a local holomorphic extension of $f$ on $U_j$. Let $F := \sum \chi_j F_j$ as a $C^{\infty}$ section of $K_X+L$.
		
	Denote the smooth hermitian metric  $h = e^{-\phi}$ for $L$.	Taking the $(n,n)$-form $g := i^{n^2} F\wedge \overline{F} \,e^{-\phi}$ as in \eqref{eq:ZZ11}, we let		
	$$
	||F||^2_\Psi :=\sup_{K \ \text{compact in} \ X} ||1_K(i^{n^2} F\wedge \overline{F} \,e^{-\phi})||_\Psi.
	$$
	In fact, this is independent of the choice of a smooth extension $F$ for the following reason \cite{NW24}, cf. \cite{D15}: let $F_1$ and $F_2$ be two such extensions. Then $F_1 - F_2 \in L^2_{\loc}$ with respect to $\phi + \Psi$. Hence the Ohsawa norm of $f$ is well defined as
	
	\begin{equation}
		||f||^2_\Psi :=	||F||^2_\Psi =\sup_{K \ \text{compact in} \ X} \left( \limsup_{t\to -\infty} \int_{\{t<\Psi<t+1\} \cap K} i^{n^2} F\wedge \overline{F} \,e^{-(\phi+\Psi)} \right).
		\end{equation}
		
		This  is the Ohsawa norm used in Theorem~\ref{th:ZZ0}.
		
		\begin{remark1}		
		
		These analytic definitions of the Ohsawa norm and the Ohsawa measure arise from  the well established  framework of the $L^2$ methods for $L^2$ extension (cf. \cite{O01}, \cite{D15}), so that the $\db$ equation for the extension problem can be solved for a given section if it has finite Ohsawa norm (cf. around (5.17) of \cite{D15}).

	\end{remark1}
	
  	Also for the convenience of readers, we have the following remark comparing the following two statements. 
	
\begin{itemize}
\item
	 (A) Theorem~\ref{th:ZZ0} with
	\item 
	  (B) The full statement of \cite[Thm. 1.2]{ZZ22} 	
\end{itemize}
so that one can identify (A) to be contained in (B). 

	\begin{remark1}\label{ZZcompare}

	\begin{enumerate}

	\item (B) requires the additional use of the so-called restricted multiplier ideal sheaf  
	$\; \II'_\Psi (h)$. Note that the sheaf morphism $$\II'_\Psi (h) \to \II'_\Psi (h) / \II(h e^{-\Psi})$$  simplifies to the classical restriction to $Y$ in Theorem~\ref{th:ZZ0}. 
	
	\item (A) is the case of (B) when $\al_0= 0$ and $R(t) = e^{-t}$ so that $C_R = 1$ and $R(\al_0)=1$.

	\item In (B), $\Psi = \psi$ is allowed to be merely locally integrable. Note that $\psi$ cannot be taken as the difference of two general quasi-psh functions due to the condition $\sup_{\Omega} \psi < \infty$. 
	\end{enumerate}

	\end{remark1}

\subsection{Singularities of the Ohsawa norm}	{\ }

In the setting of Definition~\ref{de:Ohsawa-norm}, we define the singularity of the Ohsawa norm of $\Psi$ as follows. 

\begin{definition}
(1)
 Let $p \in X$ be a point. 
We will say that the Ohsawa norm of $\Psi$ is \textbf{\emph{singular}} at $p$ if $||g||_\Psi = \infty$ for every positive $(n,n)$-form $g$ defined near $p$ with $g(p) \neq 0$. 

(2) The subset of all the singular points of the Ohsawa norm will be called the \textbf{\emph{singular locus}} of the Ohsawa norm of $\Psi$. 

\end{definition}

When $\Psi$ has log canonical singularities as in Theorem~\ref{th:ZZ0}, this is the same notion of singularity introduced in Section 1.2, which was in terms of $||f||_\Psi$ for a holomorphic function (or a section) $f$ near $p \in Y$.

	 In the following proposition, we have a relation of the singular locus (defined in the introduction) of the Ohsawa norm  in terms of the log canonical thresholds of $\Psi$.

\begin{proposition}\label{singlocus}	

Let $\Psi$ be a quasi-psh function on a complex manifold $X$. 
	Denoting by $Z$ the singular locus of the Ohsawa norm of $\Psi$, we have the relation
	$$
	 \{x\in X: c_x(\Psi)< 1\} \subset Z\subset Y_0 := \{x\in X: c_x(\Psi)\leq 1\}
	$$ 
	where $Y$ is the non-integrable locus of $e^{-\Psi}$. 
	\end{proposition}

	\begin{proof} 

	 For the first inclusion, we may locally assume that $\Psi<0$, then for every $\varepsilon>0$, we have
	$$
	\int_{\Psi<0} e^{-(1-\varepsilon) \Psi}\leq \sum_{k=2}^\infty e^{\varepsilon(1-\frac{k}2)} \int_{-\frac{k}2<\Psi<1-\frac{k}2} e^{-\Psi}
	$$
	where we omit the Lebesgue measure in the integral. 	
	Thus if $c_x(\Psi)<1$ then  $||g||_\Psi=\infty$ for all continuous $g$ with $g(x)>0$. 
	
	The second inclusion follows from the observation that, if $e^{-\Psi}$ is locally integrable on an open subset, say $U$, of $X$, then $||g||_\Psi=0$ for all continuous $g$ with compact support in $U$. 
	\end{proof}

 Now assume that $\Psi$ has log canonical singularities. In this case, the set \(\{x:c_x(\Psi)<1\}\) of Proposition~\ref{singlocus} is empty.  	
	We have the following dichotomy from \cite{K21} for the behavior of the Ohsawa norm when $\Psi$ has neat analytic singularities.
 
 \begin{theorem} \cite{K21} \label{K21}
  Let $(X, \Psi)$ be a log canonical pair with $\Psi$ a quasi-psh function with log canonical singularities.  Let $Y$ be an irreducible component of the non-integrable locus of $\Psi$. 
  Suppose that $\Psi$ has analytic singularities. Then the following hold. 
  
  \begin{enumerate}
  
  \item
 
 The Ohsawa norm is entirely singular on $Y$ if and only if the pair $(X, \Psi)$ has at least two log canonical places for $Y$. (Hence, in this case, the $L^2$ extension statement using the defining function $\Psi$ is a vacuous statement for extension from $Y$ to $X$.)

  \item
  
  The singular locus \(Z(\Psi)\cap Y\) of the Ohsawa norm is a proper subset of $Y$ if and only if the pair $(X, \Psi)$ has only one log canonical place for $Y$. 
  
  \end{enumerate}
  
 \end{theorem}
 
 Moreover in the case (2), vanishing along $Z$ is an equivalent condition for the Ohsawa norm being locally $L^2$, see \cite[Theorem 1.3]{K21}. This turns the analytic criterion in terms of the Ohsawa norm in the $L^2$ extension theorem into a geometric one.  
 
 \begin{remark1}
A special case of Theorem~\ref{K21} is when the non-integrable locus of $\Psi$ is an irreducible smooth subvariety of $X$, which is then taken as $Y$. 	
	Certainly there exist such instances of $(X, \Psi)$ both for (1) and for (2), cf. Section 2.5. 
\end{remark1}

	\section{Proofs of the main results}
	
	In this section, we prove the main  Theorem~\ref{th:main} and then  Theorem~\ref{maincoro}. 
	
	\subsection{Transition to an integral on the Newton convex body} {\ }

	We shall use the following proposition transitioning from the integral in the definition of the Ohsawa norm to an integral on the Newton convex body of $\Psi$. 
	
	\begin{proposition}\label{pr:toric} Let $\Psi$ be a $P$-toric plurisubharmonic function with isolated log canonical singularities at $0\in \mathbb D^n$. Then we have 
				\begin{equation}\label{eq:toric1}
					\limsup_{t\to -\infty} \int_{\{t<\Psi<t+1\} } \,e^{-\Psi}<\infty
							\end{equation}
						 if and only if
		\begin{equation}\label{eq:toric}
			\lim_{t\to-\infty} e^{-t}\int_{h_P(x)<t} e^{x_1+\cdots+x_n}\, dx_1\cdots dx_n < \infty
		\end{equation}
		where $h_P(x):=\sup_{\alpha\in P} \alpha\cdot x$ is the support function of $P$.
	\end{proposition}

	\begin{proof} 
		Since $\Psi(z)$ differs from $h_P(\log(|z_1|^2), \cdots, \log(|z_n|^2))$ by a bounded term, we know that
		$$
		1=c_0(\Psi)=\sup\{c\geq 0: e^{-c\,h_P(\log(|z_1|^2), \cdots, \log(|z_n|^2))}  \ \text{is integrable near $0$} \}
		$$
		and
		 \eqref{eq:toric1} is equivalent to that
		$$
		\limsup_{t\to-\infty}\int_{t<h_P(x)<t+1} e^{x_1+\cdots+x_n-h_P(x)} \, dx_1\cdots\,dx_n<\infty.
		$$
		Since $\Psi$ is locally bounded on $\mathbb D^n\setminus\{0\}$, we know that $$e^{-c\,h_P(\log(|z_1|^2), \cdots, \log(|z_n|^2))}  \ \text{is integrable near $0$}$$ if and only if
		$$
		I_c:=\int_{(-\infty, 0)^n} e^{-c\, h_P(x)} e^{x_1+\cdots+x_n}\, dx_1\cdots\,dx_n <\infty.
		$$
		By the definition of the Lebesgue integral, we have
		
\[
I_c
=
1+c\int_{-\infty}^{0}
\left(
\int_{h_P(x)<t}
e^{x_1+\cdots+x_n}\,dx_1\cdots dx_n
\right)e^{-ct}\,dt.
\]

		By the Prékopa theorem~\cite{P73}, 
		$$
		-\log \int_{h_P(x)<t} e^{x_1+\cdots+x_n}  \, dx_1\cdots\,dx_n 
		$$
		is a convex function of $t$. Thus $\sup\{c\geq 0: I_c<\infty\}=1$ implies that
		$$
		t-\log\int_{h_P(x)<t} e^{x_1+\cdots+x_n} \, dx_1\cdots\,dx_n 
		$$
		is convex  increasing in $t$ with zero derivative at $t=-\infty$. Hence 
		$$
		v(s):= \text{volume of $\{h_P(\log(|z_1|^2), \cdots, \log(|z_n|^2))<s\}$}=\int_{h_P(x)<s} e^{x_1+\cdots+x_n} \, dx_1\cdots\,dx_n 
		$$
		satisfies the assumptions in Lemma \ref{le:OT} and our proposition follows.
	\end{proof}

For the proof of Lemma~\ref{le:OT},	we shall use the relation
	$$
	\int_{t<\Psi<t+1} e^{-\Psi} = \int_{t}^{t+1} e^{-s} d \,v(s),
	$$
	where $v(s)$ denotes the volume of $\{\Psi<s\}$.
	Thus an integration by parts gives
	\begin{equation}\label{eq:ohsawa}
		\int_{t<\Psi<t+1} e^{-\Psi} =  e^{-(t+1)} v(t+1)- e^{-t} v(t) + \int_{t}^{t+1} e^{-s}  v(s)ds.
	\end{equation}
	We shall use the following calculus lemma.
	
	\begin{lemma}\label{le:OT} Assume that $e^{-s}  v(s)$ is decreasing in $s$ and
		\begin{equation}\label{eq:OT1}
			\lim_{t\to-\infty}   \frac{e^{-(t+1)} v(t+1)- e^{-t} v(t)}{e^{-(t+1)} v(t+1)}=0.
		\end{equation}
		Then the following are equivalent:
		\begin{enumerate}
			\item [(1)] $\limsup_{t\to -\infty} \int_{t<\Psi<t+1} e^{-\Psi}<\infty$; 
			\smallskip 
			\item[(2)] $\lim_{s\to-\infty} e^{-s}  v(s) <\infty$.
		\end{enumerate}	 
	\end{lemma}
	
	\begin{proof} (1) $\Rightarrow$ (2): Since $e^{-s}  v(s)$ is decreasing, \eqref{eq:ohsawa} gives
		$$
		\int_{t<\Psi<t+1} e^{-\Psi}  \geq  e^{-(t+1)} v(t+1)- e^{-t} v(t) +  e^{-(t+1)} v(t+1),
		$$
		thus \eqref{eq:OT1} gives (1) $\Rightarrow$ (2).
		
		\medskip
		
		(2) $\Rightarrow$ (1): Since $e^{-s}  v(s)$ is decreasing, \eqref{eq:ohsawa} gives
		$$
		\limsup_{t\to -\infty} \int_{t<\Psi<t+1} e^{-\Psi} \leq  e^{-t} v(t),
		$$
		which implies (2) $\Rightarrow$ (1).
	\end{proof}

	\begin{remark1}
Related normalized sublevel volumes occur in B{\l}ocki's
estimates for the Bergman kernel in terms of sublevel sets
of the pluricomplex Green function
\cite[Theorem~1]{B14}.
\end{remark1}

	\subsection{Proof of Theorem~\ref{th:main}}{\ }
	
	When $n=1$, if $\Psi$ is $P$-toric psh with isolated log canonical singularities at $0\in \mathbb D$, then $\Psi(z)-\log|z^2|$ is bounded and both (1) and (2) in Theorem~\ref{th:main} hold.
	
	 From now on,  assume that $n\geq 2$. By Proposition \ref{pr:toric}, we need to compute
	$$
	\int_{h_P(x)<t} e^{x_1+\cdots+x_n} \ \text{(we shall omit $dx_1\cdots dx_n$ sometimes)}.
	$$
	Note that $h_P(x) \leq t \Leftrightarrow x\in (-t) P^\circ$, where 
	$$
	P^\circ:=\left\lbrace x\in (-\infty,0]^n: \sup_{\alpha\in P} \alpha\cdot x \leq -1\right\rbrace
	$$
	is called the \emph{polar body} of $P$. Thus we have  
	$$
	\int_{h_P(x)<t} e^{x_1+\cdots+x_n}  =\int_{(-t) P^\circ} e^{x_1+\cdots+x_n}.
	$$
	By a change of variables $x=-ts$, we obtain
	$$
	\int_{h_P(x)<t} e^{x_1+\cdots+x_n} = (-t)^n \int_{s\in P^\circ} e^{-t(s_1+\cdots+s_n)} \, ds_1\cdots ds_n = \frac{(-t)^n}{\sqrt{n}} \int_{-\infty}^{-1} e^{-ta} f(a)\,da,
	$$
	where $f(a)$ denotes the volume of $P^\circ\cap\{s_1+\cdots+s_n=a\}$. 
	
	\begin{remark} Here $a\in (-\infty, -1)$ since
		$$
		\sup_{s\in P^\circ} s_1+\cdots+ s_n=-1  \Leftrightarrow (1,\cdots, 1)\in \partial P.
		$$
	\end{remark}

	Denote by $g(b):=f(b-1)$ the volume of $P^\circ\cap\{x_1+\cdots+x_n=b-1\}$, we obtain
	\begin{equation}\label{eq:gb}
		e^{-t}\int_{h_P(x)<t} e^{x_1+\cdots+x_n} = \frac{(-t)^n}{\sqrt{n}}\int_{-\infty}^0 e^{-tb} g(b)\, db =  \frac{(-t)^{n-1}}{\sqrt{n}} \int_{-\infty}^0 e^{s} g(-s/t)\, ds. 
	\end{equation}
	Thus $g(0)>0$ implies
	$$
	\lim_{t\to-\infty} 	e^{-t}\int_{h_P(x)<t} e^{x_1+\cdots+x_n}  = \infty.
	$$

	Now let us examine the $g(0)=0$ case. By the Brunn-Minkowski inequality, 
	$$
	h(b):=-g(b)^{\frac1{n-1}}
	$$
	is a convex function of $(-\infty, 0]$. Note that $P^\circ + (-\infty,0]^n =P^\circ$ implies that
	\begin{equation}\label{eq:hb}
		\lim_{b\to-\infty} \frac{h(b)}{b}= c_n:= \left(\text{volume of $(-\infty, 0]^n \cap \{x_1+\cdots+x_n=-1\}$}\right)^{\frac1{n-1}},
	\end{equation}
	thus $h'\geq c_n>0$ (one may compute $c_n= (\frac{\sqrt n}{(n-1)!})^{\frac1{n-1}}$) and $h$ is an increasing function of $b$. Moreover, since $h(0)=0$ and $h$ is convex, we know that the following limit
	$$
	L:=\lim_{b\to 0-} \frac{h(b)}{b} = \lim_{b\to 0-} \frac{h(b)-h(0)}{b-0}
	$$
	exists. More precisely, $L \geq c_n$ is left derivative of $h$ at $b=0$, which is equal to the increasing limit of $h(b)/b$ as $b\to 0-$.
	Now we are able to prove the following lemma.

	\begin{lemma}\label{le:newton} Denote by $g(b)$ the volume of $P^\circ\cap\{x_1+\cdots+x_n=b-1\}$. Assume that $n\geq 2$. 
	
	\begin{enumerate}[(a)]
	\item
	If $g(0)>0$ then 
		$$
		\lim_{t\to-\infty} e^{-t}\int_{h_P(x)<t} e^{x_1+\cdots+x_n} \,dx_1\cdots dx_n = \infty.
		$$
		
		\item
		If  $g(0) =0$, then the following limit 
		$$
		L:=\lim_{t\to-\infty} (-t) g(1/t)^{1/(n-1)}
		$$	
		exists, which  can be finite or infinite. Moreover, the following are equivalent:
		\begin{itemize}
			
			\item[(1)] $L<\infty$.
			
			\item[(2)] $$
			\lim_{t\to-\infty} e^{-t}\int_{h_P(x)<t} e^{x_1+\cdots+x_n} \,dx_1\cdots dx_n  < \infty.$$
			
			\item[(3)] There exist $0<\varepsilon<\frac\pi 2$ and  $v\in \mathbb R^n_{\leq 0}$ with $v_1+\cdots+v_n=-1$ such that
			$$
			P^\circ \subset \left\lbrace x\in \mathbb R^n_{\leq 0}: x_1+\cdots+x_n \leq -1, \ |\theta(v-x, 1)|<\frac\pi 2-\varepsilon \right\rbrace,
			$$
			where $\theta(v-x, 1)$ denotes the angle between $v-x$ and $(1,\cdots, 1)$.
		\end{itemize}
		
\end{enumerate}		
		
	\end{lemma}
	
	\begin{proof}[Proof of Lemma~\ref{le:newton}] The $g(0)>0$ case has been proved already. Assume that $g(0)=0$.
		
		\smallskip
		\emph{(1) implies (2)}: Since $L$ is finite, we know that
		$$
		g(b) \leq C  (-b)^{n-1}
		$$
		for negative $b$ with small $|b|$. Hence \eqref{eq:hb} implies that
		$$
		g(b) \leq C  (-b)^{n-1}
		$$
		for all negative $b$. Thus \eqref{eq:gb} implies that 
		$$
		\lim_{t\to-\infty} e^{-t}\int_{h_P(x)<t} e^{x_1+\cdots+x_n} <\infty.
		$$
		
		\smallskip
		\emph{(2) implies (1)}: By \eqref{eq:gb}, we have
		$$
		e^{-t}\int_{h_P(x)<t} e^{x_1+\cdots+x_n}  =  \frac{1}{\sqrt{n}} \int_{-\infty}^0 e^{s} g(-s/t) (-t)^{n-1}\, ds.
		$$
		Thus Fatou's lemma implies
		$$
		\lim_{t\to-\infty}  e^{-t}\int_{h_P(x)<t} e^{x_1+\cdots+x_n}  \geq  \frac{L^{n-1}}{\sqrt{n}} \int_{-\infty}^0 e^{s} (-s)^{n-1} \, ds .
		$$	
		Hence (2) implies (1).

		\smallskip
		\emph{(3) implies (1)}: Notice that $(3)$ implies that 
		$$
		g(b) \leq |\mathbb B_{n-1}| \left( \frac{-b}{\sqrt{n}} \tan \left(\frac\pi 2-\varepsilon\right)\right)^{n-1}, \ \ \forall \ b< 0,
		$$
		where $|\mathbb B_{n-1}| $ denotes the volume of the unit ball in $\mathbb R^{n-1}$. Thus we have $L<\infty$.
		
		\smallskip
		\emph{(1) implies (3)}: Notice that $(1,\cdots, 1)\in \partial P$ implies that $x_1+\cdots+x_n=-1$ is a supporting hyperplane of $P^\circ$. Thus there exists $v\in\mathbb R_{\leq 0}^n$ with $v\in P^\circ$ and $v_1+\cdots+v_n=-1$. Put
		$$
		\Theta:=\sup_{x\in P^\circ
	\setminus\{v\}} |\theta(v-x, 1)|.
		$$
		Put $\bm{e}=(-1/n,\cdots, -1/n)$. Let $e_1=(1,0,\cdots,0), \cdots, e_n=(0,\cdots,0,1)$ be the standard basis of $\mathbb R^n$. Fix a point $Q$ in the hyperplane $H:=\{x_1+\cdots+x_n=-1\}$ such that
		$$
		|\bm{e}-Q|=|\bm e|\tan\Theta=\frac{\tan\Theta}{\sqrt n}, 
		$$
		where $|\cdot|$ denotes the Euclidean norm. Denote by $C$ the convex hull of $\{-e_1, \cdots, -e_n, Q\}$ in $H$. From  $v+\mathbb R_{\leq 0}^n\subset P^\circ$ and the definition of $\Theta$, we have
		$$
		L^{n-1}=\lim_{b\to 0-} \frac{g(b)}{(-b)^{n-1}} \geq \text{volume of $C$}.
		$$
	Note that $C$ contains $Q$ and a ball of radius $r$ (depends only on $n$) around $\bm e$ in $H$, thus  there exists $c_n>0$ such that
	$$
	\text{volume of $C$} \geq c_n\, |\bm e-Q|= \frac{c_n}{\sqrt{n}}
 \tan\Theta	$$ 
		Hence
		$$
L^{n-1} \geq \frac{c_n}{\sqrt{n}}
 \tan\Theta.
		$$
		Thus if $L<\infty$ then $\tan\Theta<\infty$ and $(3)$ follows.
	\end{proof}

	\begin{proof}[Completion of the proof of Theorem~\ref{th:main}]

	Fixing a coordinate chart $z$ around $x$, from the definitions, we know that the Ohsawa norm of $\Psi$ is not singular at $x$ if and only if for any open neighborhood $U$ (inside the coordinate chart) of $x$, we have 
	\begin{equation}\label{eq:ZZ13}
		\limsup_{t\to -\infty} \int_{\{t<\Psi<t+1\} \cap U} e^{-\Psi}  < \infty,
	\end{equation}
	where we omit the Lebesgue measure in the integral.

	By Proposition \ref{pr:toric} and Lemma \ref{le:newton}, this is equivalent to the condition (3) in Lemma \ref{le:newton} for some $v$, which is equivalent to (2) in Theorem~\ref{th:main} taking $\alpha=-v$.
	\end{proof}

		\begin{proof}[Proof of Theorem~\ref{maincoro}]
	
	Let $P$ be a Newton convex body in $\RR^n_{\ge 0}$ such that
	
	\begin{itemize}
	\item the boundary $\partial P$ is locally a smooth real hypersurface (but not a hyperplane) near the point $Q:= (1, \ldots, 1) \in \partial P$ with the tangent space $H$ at $Q$, and 
	
	\item the complement of $P$ in $\RR^n_{\ge 0}$  is bounded.
	
	\end{itemize}
	
	As is well known (cf.  \cite{R13}, \cite{KS20}), given an arbitrary Newton convex body $P \subset \RR^n_{\ge 0}$, there exists a toric psh function $\Psi$ (non-uniquely) having $P$ as its Newton convex body. One particular choice is $\Psi:=\phi_P$ in \eqref{eq:phiP}. Due to the second bullet, $\Psi$ has isolated singularity at $0$.  We claim that this is the $\Psi$ we want. 
	
	Due to the first bullet and Theorem~\ref{th:main}, the Ohsawa norm of $\Psi$ is singular at $0$. Then it is  entirely singular having $Y_0 = Z = \{ 0 \}$ since $\Psi$ has isolated singularity at $0$.

	By Theorem~\ref{theonly}, the only possible log canonical place is a monomial valuation.  By Proposition~\ref{suphyp}, there is only one monomial valuation that can be (and is) a log canonical place, namely the one corresponding to the tangent hyperplane $H$ of $\partial P$ at $Q$, since $H$ is the unique supporting hyperplane through $Q$ for $P$.  This completes the proof.
		\end{proof} 
		
		Finally, we remark that instances described in Theorem~\ref{maincoro} with non-isolated singularities are obtained by using toric psh functions of the form $$\Phi (z_1, \ldots, z_n, \ldots, z_{n+m}) := \Psi (z_1, \ldots, z_n)$$ for $m \ge 1$. 

\begin{proof}[Proof of Corollary~\ref{notvoid}]
Similarly to the proof of Theorem~\ref{maincoro}, it suffices to choose a Newton convex body $P \subset \RR^n_{\ge 0}$ satisfying (2) of Theorem~\ref{th:main} such that the corresponding toric psh function $\phi_P$ has non-analytic singularities. For this purpose, as in Remark~\ref{noteven},  choose $P$ such that 
 the boundary $\partial P$ of $P$ is a hyperplane in a small open neighborhood of the point $q := (1, \cdots, 1)$, however not even polyhedral somewhere away from $q$. Then the corresponding toric psh function $\phi_P$ has non-analytic singularities. Indeed, if $\phi_P$ has analytic singularities, then $\phi_P-c \log \abs{\mfa}$ is bounded for a monomial ideal $\mfa$ and $c > 0$ due to \cite[Prop. 3.1]{AS23}. Then $P$ should be a positive multiple of the Newton polyhedron associated to $\mfa$, which contradicts the fact that $P$ is not polyhedral somewhere away from $q$. 
\end{proof}

	\footnotesize

\bibliographystyle{amsplain}

\begin{thebibliography}{widest-}


\renewcommand{\baselinestretch}{1.0}
		
		
		 
\bibitem[AS23]{AS23} J. An and H. Seo, \emph{On equisingular approximation of plurisubharmonic functions}, J. Math. Anal. Appl. 521 (2023), no. 2, Paper No. 126987, 16 pp.



\bibitem[BFJ08]{BFJ08} S. Boucksom, C. Favre and M. Jonsson, \textit{Valuations and plurisubharmonic singularities}, Publ. Res. Inst. Math. Sci, 44 (2008), no.2, 449-494.
		
		\bibitem[BM97]{BM97} E. Bierstone and P. Milman, \emph{Canonical desingularization in characteristic zero by blowing up the maximum strata of a local invariant}, Invent. Math. 128 (1997), no. 2, 207–302. 
		
		

\bibitem[B12]{B13} Z. Błocki, \emph{Suita conjecture and the Ohsawa-Takegoshi extension theorem}, Invent. Math. 193 (2013), no. 1, 149–158.


\bibitem[B14]{B14} Z. Błocki, \emph{A lower bound for the Bergman kernel and the Bourgain–Milman inequality}, Lecture Notes in Mathematics 2116 (2014), 53–63.


\bibitem[CWW15]{CWW15} B.-Y. Chen, J. Wu and X. Wang, \emph{Ohsawa-Takegoshi type theorem and extension of plurisubharmonic functions}, Math. Ann. 362 (2015), no. 1-2, 305–319.

\bibitem[D92]{D92} J.-P. Demailly, \emph{Regularization of closed positive currents and intersection theory},  J. Algebraic Geom. 1 (1992), 361–409.


\bibitem[D00]{D00} J.-P. Demailly, \emph{On the Ohsawa-Takegoshi-Manivel $L^2$ extension theorem}, Complex analysis and geometry (Paris, 1997), 47–82, Progr. Math., 188, Birkhäuser, Basel, 2000. 


\bibitem[D11]{D11} J.-P. Demailly, \emph{Analytic methods in algebraic geometry}, Surveys of Modern Mathematics, 1. International Press, Somerville, MA; Higher Education Press, Beijing, 2012. 


\bibitem[D15]{D15} J.-P. Demailly, \textit{Extension of holomorphic functions defined on non reduced analytic subvarieties}, The legacy of Bernhard Riemann after one hundred and fifty years. Vol. I, 191 - 222, Adv. Lect. Math., 35.1, Int. Press, Somerville, MA, 2016. 



\bibitem[D16]{D16}
J.-P. Demailly,
\emph{Extension of holomorphic functions defined on non reduced
analytic subvarieties}, Lecture at the Symposium \emph{G\'eom\'etrie Alg\'ebrique},
Coll\`ege de France, Paris, October 4, 2016.

\url{https://www.college-de-france.fr/en/agenda/symposium/algebraic-geometry/extension-of-holomorphic-functions-defined-on-non-reduced-analytic-subvarieties}




\bibitem[DK01]{DK01} J.-P. Demailly and J. Koll\'ar, \emph{Semi-continuity of complex singularity exponents and Kähler-Einstein metrics on Fano orbifolds}, Ann. Sci. École Norm. Sup. (4) 34 (2001), no. 4, 525–556.


		
	
	\bibitem[FJ05]{FJ05} C. Favre and M. Jonsson, \emph{Valuations and multiplier ideals},
  J. Amer. Math. Soc. 18 (2005), no. 3, 655–684.
		
		\bibitem[G12]{G12} H. Guenancia, {\it Toric plurisubharmonic functions and analytic adjoint ideal sheaves}, Math.
		Z. {\bf 271} (2012), 1011--1035.
		
		
		\bibitem[GZZ12]{GZZ12} Q. Guan, X. Zhou and L. Zhu, \emph{On the Ohsawa-Takegoshi $L^2$ extension theorem and the Bochner-Kodaira identity with non-smooth twist factor}, J. Math. Pures Appl. (9) 97 (2012), no. 6, 579–601. 
		
		\bibitem[H64]{H64} H. Hironaka, \emph{Resolution of singularities of an algebraic variety over a field of characteristic zero}, I, II. Ann. of Math. (2) 79 (1964), 109–203; ibid. (2) 79 (1964), 205–326.


		\bibitem[JM12]{JM12} M. Jonsson and M. Musta\c{t}ă, {\it Valuations and asymptotic invariants for sequences
			of ideals}, Ann. Inst. Fourier (Grenoble) {\bf 62} (2012), 2145--2209.
		
		\bibitem[JM14]{JM14} M. Jonsson and M. Musta\c{t}ă, \emph{An algebraic approach to the openness conjecture of Demailly and Kollár}, J. Inst. Math. Jussieu 13 (2014), no. 1, 119–144.


		
		\bibitem[K10]{K10} D. Kim, \emph{$L^2$ extension of adjoint line bundle sections}, Ann. Inst. Fourier (Grenoble) 60 (2010), no. 4, 1435 - 1477. 
		
		\bibitem[K21]{K21} D. Kim, \emph{$L^2$  extension of holomorphic functions for log canonical pairs}, J. Math. Pures Appl. (9) 177 (2023), 198-213.
		
		\bibitem[KK23]{KK23} D. Kim and J. Koll\'ar, \emph{Log canonical singularities of plurisubharmonic functions}, arXiv:2312.16140, to appear in  Annales de l'Institut Fourier. 
		

\bibitem[KS20]{KS20} D. Kim and H. Seo, \emph{Jumping numbers of analytic multiplier ideals}, Ann. Polon. Math. 124 (2020), no. 3, 257–280.


\bibitem[KS21]{KS21} D. Kim and H. Seo, \emph{On $L^2$ extension from singular hypersurfaces},  Math. Z. 303 (2023), no. 4, Paper No. 89, 21 pp.



\bibitem[Ks93]{Ks93} C. Kiselman, \emph{Plurisubharmonic functions and their singularities}, Complex potential theory (Montreal, PQ, 1993), 273–323, NATO Adv. Sci. Inst. Ser. C: Math. Phys. Sci., 439, Kluwer Acad. Publ., Dordrecht, 1994.


\bibitem[Ks94]{Ks94} C. Kiselman, \emph{Attenuating the singularities of plurisubharmonic functions}, Ann. Polon. Math. 60 (1994), no. 2, 173-197. 

\bibitem[Ko13]{Ko13} J. Koll\'ar, \emph{Singularities of the minimal model program}, Cambridge Tracts in Mathematics, 200. Cambridge University Press, Cambridge, 2013. 

\bibitem[LR99]{LR99} P. Lelong and A. Rashkovskii, \emph{Local indicators for plurisubharmonic functions}, J. Math. Pures Appl. (9) 78 (1999), no. 3, 233–247.


\bibitem[M93]{M93} L. Manivel, \emph{Un th\'eor\`eme de prolongement $L^2$ de sections holomorphes d'un fibr\'e hermitien}, Math. Z. 212 (1993), no. 1, 107–122. 





\bibitem[NW22]{NW22}  T. Nguyen and X. Wang, \emph{On a remark by Ohsawa related to the Berndtsson-Lempert method for $L^2$-holomorphic extension}, Ark. Mat. 60 (2022), no. 1, 173–182.


\bibitem[NW24]{NW24}  T. Nguyen and X. Wang, \emph{A Hilbert bundle approach to the sharp strong openness theorem and the Ohsawa-Takegoshi extension theorem}, Convex and complex: perspectives on positivity in geometry, 99–123, Contemp. Math., 810, Amer. Math. Soc., Providence, RI, 2025. 




\bibitem[O94]{O94} T. Ohsawa, \emph{On the extension of $L^2$ holomorphic functions. IV. A new density concept}, Geometry and analysis on complex manifolds, 157–170, World Sci. Publ., River Edge, NJ, 1994.

		\bibitem[O01]{O01} T. Ohsawa, {\it On the extension of $L^2$-holomorphic functions. V. Effect of generalization}, Nagoya Math J. {\bf 161}
		(2001), 1--21.
		
		
		\bibitem[O15]{O15} T. Ohsawa, \emph{Application and simplified proof of a sharp L2 extension theorem}, Nagoya Math. J. 220 (2015), 81–89.
		
		\bibitem[OT87]{OT87} T. Ohsawa and K. Takegoshi, \emph{On the extension of $L\sp 2$ holomorphic functions}, Math. Z. 195 (1987), no. 2, 197-204. 
		
		
		
	\bibitem[P73]{P73}	A. Prékopa, \emph{On logarithmic concave measures and functions}, Acta Sci. Math. (Szeged) 34 (1973), 335–343.


		\bibitem[R13]{R13}
A. Rashkovskii, \textit{Multi-circled singularities, Lelong numbers, and integrability index}, J. Geom.
Anal. \textbf{23} (2013), no. 4, 1976-1992.

\bibitem[Sc13]{Sc13}
R. Schneider, \textit{Convex Bodies: The Brunn–Minkowski Theory}, Cambridge University Press, Cambridge, 2013.

\bibitem[S02]{S02} Y.-T. Siu, \emph{Extension of twisted pluricanonical sections with plurisubharmonic weight and invariance of semipositively twisted plurigenera for manifolds not necessarily of general type}, Complex geometry (G\"ottingen, 2000), 223-277, Springer, Berlin, 2002.




\bibitem[X20]{X20} C. Xu, \emph{A minimizing valuation is quasi-monomial}, Ann. of Math. (2) 191 (2020), no. 3, 1003–1030.

		
		\bibitem[ZZ22]{ZZ22}  X. Zhou and L. Zhu, {\it Extension of cohomology classes and holomorphic sections defined on subvarieties}, J. Algebraic Geom. 31 (2022), no. 1, 137–179.  
		
		
	\end{thebibliography}

	\quad

\normalsize

\noindent \textsc{Dano Kim}

\noi Department of Mathematical Sciences and Research Institute of Mathematics

\noi Seoul National University, 08826  Seoul, Korea

\noi Email address: kimdano@snu.ac.kr

\quad

\noi \textsc{Xu Wang}

\noi Department of Mathematical Sciences

\noi Norwegian University of Science and Technology, No-7491 Trondheim, Norway

\noi Email address: xu.wang@ntnu.no

\end{document}